\documentclass{amsart}
\usepackage{amssymb,amsmath,darinams}
\usepackage[all]{xy}
\newcommand{\opn}[2]{\newcommand{#1}{\operatorname{#2}}}
\opn{\LS}{LS}
\opn{\Bun}{Bun}
\opn{\Shatz}{Shatz}
\opn{\Pic}{Pic}
\opn{\Conn}{Conn}
\opn{\OOper}{Oper}
\newcommand{\Oper}{\OOper'}
\opn{\Op}{Op}
\opn{\edim}{exp.dim}
\newcommand{\F}{\mathbb{F}}
\newcommand{\K}{\mathbb{K}}
\newcommand{\OO}{\mathbb{O}}
\newcommand{\kk}{\Bbbk}
\newcommand{\frg}{\mathfrak{g}}
\newcommand{\frb}{\mathfrak{b}}
\newcommand{\frp}{\mathfrak{p}}
\newcommand{\frn}{\mathfrak{n}}
\newcommand{\frt}{\mathfrak{t}}

\newcommand{\fru}{\mathfrak{u}}
\newcommand{\ffg}[1]{\frg^{(#1)}}
\newcommand{\sfv}{\mathsf{v}}

\makeatletter
\newcommand{\eqnum}{\refstepcounter{equation}\textup{\tagform@{\theequation}}}
\makeatother

\begin{document}
\title{Irreducible connections admit generic oper structures}
\author{Dima Arinkin}
\address{University of Wisconsin, 480 Lincoln Drive, Madison, WI 53706, USA}
\email{arinkin@math.wisc.edu}
%\date{02/16}
\begin{abstract}Let $G$ be a connected reductive group and $X$ be a smooth curve over an algebraically closed field of characteristic zero. We show that every meromorphic
$G$-connection on $X$ admits a possibly degenerate oper structure; in particular,
every irreducible meromorphic $G$-connection admits a generic oper structure.
\end{abstract}
\maketitle

\section{Introduction}\label{sec:intro}

\subsection{Main result}
Let $\kk$ be an algebraically closed field of characteristic zero. Let $G$ be a connected reductive group over $\kk$, and let
$X$ be a smooth connected projective curve over $\kk$. Let $\F=\kk(X)$ be the field of rational functions on $X$.
Thus, $\F/\kk$ is a finitely generated field extension of transcendence degree one.

Our goal is to show that every meromorphic
$G$-connection on $X$ admits a `possibly degenerate oper structure'; in particular,
every irreducible meromorphic $G$-connection admits a generic oper structure. Explicitly, the statement is
formulated in terms of $\F$ as follows.

Fix a Borel subgroup $B\subset G$, and let $N\subset B$ be its
unipotent radical. We denote the Lie algebras of $G\supset B\supset N$ 
by $\frg\supset\frb\supset\frn$, respectively. Now put 
\[\ffg{-1}:=\{x\in\frg:[x,\frn]\subset\frb\}\subset\frg.\] 
Clearly, $\ffg{-1}$ is a vector space containing $\frb$; the quotient $\ffg{-1}/\frb$ naturally splits as a direct sum of root spaces
\[\ffg{-1}/\frb=\bigoplus_{\alpha\in S}\frg_{-\alpha}.\]
Here $S$ is the set of the simple positive roots of $G$.

Let $\Omega_{\F/\kk}$ be the space of K\"ahler differentials of $\F$ over $\kk$. It is a one-dimensional vector space
over $\F$. Let 
\[\Conn_G(\F):=\{\nabla=d+A: A\in\frg\otimes_\kk\Omega_{\F/\kk}\}\]
be the set of $G$-connections on the trivial $G$-bundle over $\spec(\F)$.
The group $G(\F)$ acts on $\Conn_G(\F)$ by gauge transformations;
we write this action as
\[\nabla\mapsto \Ad(g)\nabla=d+(\Ad(g)A-(dg)g^{-1})\qquad (\nabla\in\Conn_G(\F),g\in G(\F)).\]
Denote by $\sim$ the gauge equivalence relation on $\Conn_G(\F)$. Thus,
$\nabla\sim\nabla'$ if and only if there exists $g\in G(\F)$ such that $\nabla'=\Ad(g)\nabla$.

Put
\[\OOper_G(\F):=\{\nabla=d+A:A\in\ffg{-1}\otimes_\kk\Omega_{\F/\kk}\}\subset\Conn_G(\F).\] 
We call $\OOper_G(\F)$ the set of possibly degenerate opers (with respect to the
trivial $B$-structure). 

The main result of this paper is the following theorem.

\begin{THEOREM} \label{th:main}
For any $\nabla\in\Conn_G(\F)$, there exists $\nabla'\in\OOper_G(\F)$ such that $\nabla'\sim\nabla$.
\end{THEOREM}  

The proof of Theorem~\ref{th:main} relies on degeneration of the moduli space of bundles with connections
into the moduli space of Higgs bundles. In Section~\ref{sec:conventions}, we reformulate Theorem~\ref{th:main}
in terms of $G$-bundles on the curve $X$, while also extending it to Higgs fields (Theorem~\ref{th:mainbundle}).

\subsection{Generic oper structures}
We say that
$\nabla=d+A\in\OOper_G(\F)$ 
is \emph{(generically) non-degenerate} if the image of $A$ in $\frg_{-\alpha}\otimes_\kk\Omega_{\F/\kk}$
is non-zero for all simple roots $\alpha\in S$. Denote by 
\[\Oper_G(\F)\subset\OOper_G(\F)\]
the subset of non-degenerate opers. The case of $G=\GL(n)$ (Section~\ref{sec:classical groups}) suggests that
Theorem~\ref{th:main} should still hold if we require that $\nabla'\in\Oper_G(\F)$. However, we do not
know how to prove this stronger statement.

For irreducible connections, the stronger form of Theorem~\ref{th:main} is equivalent to Theorem~\ref{th:main}. Recall that $\nabla\in\Conn_G(\F)$ is \emph{reducible} if there exists a proper parabolic subgroup $P\subset G$
and $\nabla'\in\Conn_P(\F)$ such that $\nabla'\sim\nabla$. Here
\[\Conn_P(\F)=\{d+A:A\in\frp\otimes_\kk\Omega_{\F/\kk}\}\subset\Conn_G(\F),\]
and $\frp$ is the Lie algebra of $P$.

\begin{corollary} \label{co:operirreducible}
Suppose $\nabla\in\Conn_G(\F)$ is irreducible. Then there exists $\nabla'\in\Oper_G(\F)$ such that $\nabla'\sim\nabla$.
\end{corollary}
\begin{proof} By Theorem~\ref{th:main}, there is $\nabla'\in\OOper_G(\F)$ such that $\nabla'\sim\nabla$.
Since $\nabla$ is irreducible, $\nabla'\in\Oper_G(\F)$.
\end{proof}

\subsection{Classical groups}\label{sec:classical groups}
Suppose $G=\GL(n)$. In this case, Theorem~\ref{th:main} is well known; 
it can be viewed as the correspondence between linear ordinary
differential equations of order $n$ and systems of $n$ linear ordinary differential equations of order one. 

To make the argument explicit, let us fix a non-constant function $t\in\F$. A connection $\nabla\in\Conn_G(\F)$
defines a first order differential operator
\[\nabla_{\frac{d}{dt}}=\frac{d+A}{dt}:\F^n\to\F^n.\]
The operator $\nabla_{\frac{d}{dt}}$ admits a cyclic vector $\vec{v}\in\F^n$ by \cite[Theorem~1]{Katz}. Using the
change-of-basis matrix
\[g=\begin{pmatrix}\vec{v}&\nabla_{\frac{d}{dt}}\vec{v}&\dots&\left(\nabla_{\frac{d}{dt}}\right)^{n-1}\vec{v}\end{pmatrix},\]
we obtain the `Frobenius normal form'
\[\Ad(g^{-1})(\nabla)=d+\begin{pmatrix}
0&0&\dots&0&*\\
dt&0&\dots&0&*\\
0&dt&\dots&0&*\\
\dots&\dots&\dots&\dots&\dots\\
0&0&\dots&dt&*
\end{pmatrix}.\]
In particular, $\Ad(g^{-1})(\nabla)\in\Oper_G(\F)$.

Recently, D.~Kazhdan and T.~Schlank have extended this result to the case of classical group $G$ (\cite{KS}). Moreover, they prove that if $G$ is classical, the space of generic oper structures on a $G$-local system is
contractible. We do not know whether the contractibility holds for exceptional $G$.

\subsection{Connections on the formal disk}
The question addressed in this paper has a local counterpart. Namely, let us replace $\F$ with the field of 
formal Laurent series $\K=\kk((t))$. The sets $\Conn_G(\K)\supset\OOper_G(\K)\supset\Oper_G(\K)$ still make sense in
this setting, and a version of Theorem~\ref{th:main} in this case is due to E.~Frenkel and X.~Zhu:

\begin{THEOREM}[{\cite[Theorem~1]{FZ}}] \label{th:FZ}
For any $\nabla\in\Conn_G(\K)$, there exists $g\in G(\K)$ such that 
\[\Ad(g)\nabla\in\Oper_G(\K).\]
\end{THEOREM}

It is not clear whether Theorem~\ref{th:FZ} implies Theorem~\ref{th:main}. On the other hand, it is not hard to 
derive Theorem~\ref{th:FZ} from Theorem~\ref{th:main}; we sketch the argument in Section~\ref{sec:formaldisk}.

\subsection{Relation to the geometric Langlands conjecture}
Our interest in generic oper structures is in part motivated by the geometric Langlands conjecture. In \cite{G},
D.~Gaitsgory outlines a plan of proof of the geometric Langlands conjecture. The plan contains several gaps, whose size depends on the group $G$ (smallest for $G=\GL(2)$, larger for $G=\GL(n)$, largest for arbitrary $G$). In particular, two crucial steps of the plan (Conjectures~8.2.9 and 10.2.8 of \cite{G}) are described in \cite{G}
as still mysterious in the case of an arbitrary group $G$. In fact, Conjecture 10.2.8 is equivalent to Corollary~\ref{co:operirreducible}; thus, the present paper fills in one of the two major gaps.

\subsection{Organization}
This paper is organized as follows. 

In Section~\ref{sec:conventions}, we define the basic objects that we work with: $G$-bundles with $K$-connections
and $K$-opers. ($K$-connections can be viewed as P.~Deligne's $\lambda$-connections with poles.) We then state
the main result in these terms (Theorem~\ref{th:mainbundle}). In Section~\ref{sc:moduli}, we consider the moduli
stacks of $K$-connections and $K$-opers and reduce Theorem~\ref{th:main} to three geometric propositions,
whose proofs occupy Sections~\ref{sc:closedness}, \ref{sc:good}, and \ref{sc:witness}. Finally, in Section~\ref{sec:formaldisk}, we derive Theorem~\ref{th:FZ} from Theorem~\ref{th:main}.

\subsection{Acknowledgements}
I am very grateful to V.~Drinfeld and D.~Gaitsgory for numerous stimulating discussions. This work was supported in
part by the NSF grant DMS-1452276.

\section{Conventions}\label{sec:conventions}

\subsection{Schemes and stacks}
Fix once and for all an algebraically closed ground field $\kk$ of characteristic zero. All schemes and stacks
that appear in this note are locally of finite type over $\kk$.

Some of the moduli stacks that we work with are actually dg-stacks. This applies to the moduli stack of
bundles with $K$-connections $\LS_{G,K}$ and the moduli stack of $K$-valued opers $\Op_{G,K}$. However, their dg-structure is relatively mild: the stacks are quasi-smooth (which is the dg-version of the notion of local complete
intersection). 
Moreover, the essential part of the argument concerns classical points of the stacks; that is, the points where the dimension and the expected dimension of the stacks are equal (and therefore, the stacks are local complete intersections in the usual sense). We use no deep results on dg-stacks.

\subsection{Bundles and twists}

Let $X$ be a scheme, and let $G$ be an algebraic group. Given a $G$-bundle $\cF$ on $X$ and a scheme $Z$ equipped 
with an action of $G$, we denote by
\[Z_{\cF}:=(Z\times\cF)/G\]
the $\cF$-twisted form of $Z$. Here $G$ acts on $Z\times\cF$ diagonally. We have a natural morphism $Z_\cF\to X$, which is a locally
trivial fibration with fiber $Z$. In particular, if $V$ is a finite-dimensional representation of $G$, the twist $V_\cF$ is a vector
bundle on $X$.

Let $H\subset G$ be a closed subgroup. For any $G$-bundle $\cF$, sections of the morphism $(G/H)_{\cF}\to X$ are in bijection with reductions
of $\cF$ from $G$ to $H$. Recall that a reductions of $\cF$ to $H$ is a closed $H$-invariant subscheme $\cF_H\subset\cF$ which is an
$H$-bundle over $X$.

\subsection{Reductive groups}
Let $G$ be a connected reductive group (over $k$). Fix a Borel subgroup $B\subset G$, and denote by $N\subset B$ its
unipotent radical and by $T=B/N$ the Cartan group of $G$. We denote the Lie algebras of $G\supset B\supset N$ and of $T$
by $\frg\supset\frb\supset\frn$ and $\frt$, respectively.

We denote by $\pi_1(G)$ the algebraic fundamental group of $G$, that is, the quotient of the cocharacter lattice
$\Hom(\gm,T)$ by the coroot lattice. More generally, for any algebraic group $G$, we put $\pi_1(G)=\pi_1(G_0/U)$,
where $G_0\subset G$ is the identity component and $G_0/U$ is its maximal reductive quotient.  

The Lie algebra $\frg$ carries a filtration by vector spaces $\ffg{k}$ such that 
\begin{align*}
\ffg{0}&=\frb,\\
\ffg{1}&=\frn,&\ffg{-1}&=\{a\in\frg:[a,\frn]\subset\frb\},\\
\dots&&\dots\\
\ffg{k+1}&=[\frn,\ffg{k}],&\ffg{-k-1}&=\{a\in\frg:[a,\frn]\subset\ffg{-k}\},\\
\dots&&\dots
\end{align*}
(The filtration is given by the principal $\mathfrak{sl}(2)$ in $\frg$.) Note that
the quotient $\ffg{-1}/\frb$ naturally splits as a direct sum
\begin{equation}\label{eq:-1splitting}
\ffg{-1}/\frb=\bigoplus_{\alpha\in S}\frg_{-\alpha}.
\end{equation}
Here $S$ is the set of the simple positive roots of $G$. 
 
\subsection{$K$-connections} As above, let $X$ be a scheme and let $G$ be an algebraic group.
Let $K$ be a vector bundle on $X$ together with a morphism $\lambda_K:\Omega_X\to K$. 
(Here and everywhere else in the paper, we use the same notation for a vector bundle and its sheaf of sections.) Let $d_K$ be the composition
\[d_K:=\lambda_K\circ d:\cO_X\to\Omega_X\to K;\]
it is a $K$-valued derivation of $\cO_X$. It now makes sense to talk about $K$-connections on $G$-bundles on $X$.

\begin{definition} Let $\cF$ be a $G$-bundle on $X$; denote the projection $\cF\to X$ by $p_\cF$. A \emph{$K$-connection} on $\cF$ is a $G$-equivariant morphism $\nabla:\Omega_\cF\to p_\cF^*K$ such that the composition
\[p_\cF^*\Omega_X\overset{dp_\cF}\to\Omega_\cF\overset\nabla\to p_\cF^*K\]
equals $p_\cF^*\lambda_K$. 
\end{definition}

\begin{example} For $G=\GL(n)$, we can use the equivalence
\[\cF\mapsto (\mathbb A^n)_\cF\]
between $G$-bundles and rank $n$ vector bundles over $X$. Under this equivalence, a $K$-connection on a 
vector bundle $E$ over $X$ 
is simply a $\kk$-linear morphism 
\[\nabla:E\to E\otimes K\]
satisfying the Leibnitz rule
\[\nabla(fs)=f\nabla(s)+s\otimes d_K(f)\qquad (f\in\cO_S,s\in E).\]
\end{example}

\begin{example} \label{ex:matrix}
Given a local trivialization of the $G$-bundle $\cF$, we can write a $K$-connection $\nabla$ on $\cF$ as 
a differential operator
\[d+A\qquad A\in \frg\otimes K.\]
Here $A$ is the connection matrix of $\nabla$ with respect to the trivialization. 
Passing to a different trivialization acts on $A$ as the gauge transform:
\[A\mapsto (\Ad(g)A-(d_K(g))g^{-1}),\]
where $g:X\to G$ is the gauge change.
\end{example}

\begin{remarks}
\begin{enumerate}
\item 
$K$-connections can be viewed as meromorphic $\lambda$-connections; 
conversely, $\lambda$-connections
(introduced by P.~Deligne) are $K$-connections for $K=\Omega_X$ and $\lambda_K\in\kk$ being a scalar.

\item In this paper, we consider connections on a (smooth) curve $X$ with $K$ being a line bundle.  
We then have the following options for $K$ and $\lambda_K$:
\begin{itemize}
\item $K=\Omega_X$, and $\lambda_K$ is the identity map:  $K$-connections are `usual' connections.

\item More generally, suppose $K=\Omega_X(D)$ for an effective divisor $D$ on $X$, and $\lambda_K$ 
is the natural embedding: $K$-valued connections are meromorphic connections whose pole is bounded by $D$.

\item $K$ is any line bundle and $\lambda_K=0$: $K$-connections are $K$-valued Higgs fields, that is, 
sections of the vector bundle $\frg_{\cF}\otimes K$. 
\end{itemize}

\item It is well known that any $G$-bundle on a smooth curve over an algebraically closed field is locally trivial in the Zariski topology (see \cite[Remark~2.b]{DS} for references). Thus, $K$-connections on curves can be described locally by their matrices, as in Example~\ref{ex:matrix}.

\item Since we are interested in $K$-connections on curves only, we ignore the notion of curvature of $K$-connections.
\end{enumerate}
\end{remarks}

Suppose now that $G$ is a connected reductive group.

\begin{definition} A $K$-connection $\nabla$ on a $G$-bundle $\cF$ is \emph{irreducible} if $\cF$ does not admit a $\nabla$-invariant $P$-structure for any proper parabolic subgroup $P\subset G$. Equivalently, the pair $(\cF,\nabla)$
is \emph{irreducible} if it is not induced from a $P$-bundle with a $K$-connection for any proper parabolic subgroup $P\subset G$.
\end{definition}

Note that if $\cF_P$ is a $P$-structure on a $G$-bundle $\cF$, then a $K$-connection $\nabla$ induces a section
\[\nabla_{/\cF_P}\in H^0(X,(\frg/\frp)_{\cF_P}\otimes K),\] 
where $\frp\subset\frg$ is the Lie algebra of $P$. The $P$-structure $\cF_P$ is $\nabla$-invariant if and only $\nabla_{/\cF_P}=0$.

\subsection{Opers}
As above, let $G$ be a connected reductive group, $X$ a scheme, and $K$ be a vector bundle on $X$ together with a morphism 
$\lambda_K:\Omega_X\to K$. (We are only interested in the case when $X$ is a curve and $K$ is a line bundle.)

Let $\cF$ be a $G$-bundle on $X$, and let $\nabla$ be a $K$-valued connection on $\cF$. 

\begin{definition}\label{def:oper}
 A \emph{possibly degenerate oper structure} on $(\cF,\nabla)$ is a reduction $\cF_B$ of 
$\cF$ to the Borel subgroup $B\subset G$ such that for any local trivialization of $\cF$ that is compatible with
$\cF_B$, the matrix of the connection $\nabla$ is a section of $\frg^{(-1)}\otimes K$.

From now on, we omit the word `possibly' and refer to $\cF_B$ as a `degenerate oper structure', and to a triple $(\cF,\nabla,\cF_B)$ as a `degenerate $K$-oper'.
\end{definition}

We can restate Definition~\ref{def:oper} as follows. Given the reduction $\cF_B$ of $\cF$, the $K$-connection $\nabla$ induces a section
\[\nabla_{/\cF_B}\in H^0(X,(\frg/\frb)_{\cF_B}\otimes K).\]
The triple $(\cF,\nabla,\cF_B)$ is a degenerate $K$-oper if and only if
\[\nabla_{/\cF_B}\in H^0(X,(\ffg{-1}/\frb)_{\cF_B}\otimes K)\subset H^0(X,(\frg/\frb)_{\cF_B}\otimes K).\]

\begin{remark} Suppose $(\cF,\nabla,\cF_B)$ is a degenerate $K$-oper. The decomposition \eqref{eq:-1splitting} 
yields an isomorphism
\[H^0(X,(\ffg{-1}/\frb)_{\cF_B}\otimes K)\simeq\bigoplus_{\alpha\in S}H^0\left(X,(\frg_{-\alpha})_{\cF_B}\otimes K)\right).\]
We can therefore write $\nabla_{/\cF_B}=\sum_{\alpha\in S}\nabla_{-\alpha}$ for
\[\nabla_{-\alpha}\in H^0\left(X,(\frg_{-\alpha})_{\cF_B}\otimes K)\right).\]
We say that $(\cF,\nabla,\cF_B)$ is generically a non-degenerate $K$-oper if $\nabla_{-\alpha}\ne 0$ for all $\alpha\in S$. Here the word `generically' refers to the possibility that the components $\nabla_{-\alpha}$ have zeroes on $X$. 
Note in particular that if the underlying $K$-connection $\nabla$ is irreducible, then any degenerate oper structure satisfies this condition. 
\end{remark}

\subsection{Main result}
We can now state the main result of this paper using the language of bundles with $K$-connections. 
Let $X$ be a smooth projective curve, $K$ a line bundle on $X$ equipped with a morphism $\lambda_K:\Omega_X\to K$,
and $G$ a connected reductive group.

\begin{theorem} \label{th:mainbundle}
Any $G$-bundle with $K$-connection on $X$ admits a degenerate oper structure.
\end{theorem}

Clearly, Theorem~\ref{th:mainbundle} implies Theorem~\ref{th:main}. Indeed, any $\nabla\in\Conn_G(\F)$ for $\F=\kk(X)$ can be viewed as a $K$-connection on the trivial $G$-bundle for $K=\Omega_X(D)$, where $D$ is the divisor of poles of $\nabla$. 

\section{Moduli of connections and opers}\label{sc:moduli}
We now turn to the proof of Theorem~\ref{th:mainbundle}. In this section, we study the moduli stacks of $K$-connections and $K$-opers, and reduce Theorem~\ref{th:mainbundle} to three
claims, whose proofs occupy Sections~\ref{sc:closedness}, \ref{sc:good}, and \ref{sc:witness}.

From now on, we fix a connected reductive group $G$, an smooth connected projective curve $X$, and a line bundle $K$ on $X$ equipped with a morphism $\lambda_K:\Omega_X\to K$. Note that the assumption that $X$ is connected
does not restrict the generality of Theorem~\ref{th:mainbundle}.

\subsection{Moduli of bundles}
Let $\Bun_G$ be the stack of $G$-bundles on $X$. Recall that $\Bun_G$ is a smooth stack of pure dimension $(g-1)\dim(G)$. Its connected components are in bijection with $\pi_1(G)$; we denote the component corresponding
to $d\in\pi_1(G)$ by $\Bun^d(G)$. If $\cF\in\Bun^d(G)$, we write $d=\deg(\cF)$.

\subsection{Moduli of connections}
Denote by $\LS_{G,K}$ 
the stack of pairs $(\cF,\nabla)$, where $\cF$ is a $G$-bundle and $\nabla$ is a $K$-connection on $\cF$. Generally speaking, $\LS_{G,K}$ is a dg-stack; however, it is classical if $\deg(K)$ is large enough (see Proposition~\ref{prop:good}). 

The tangent complex to $\LS_{G,K}$ is given by the de Rham complex
\[T_{(\cF,\nabla)}\LS_{G,K}=R\Gamma(X,\frg_{\cF}\overset{\ad(\nabla)}\to\frg_{\cF}\otimes K)[1]\qquad ((\cF,\nabla)\in\LS_{G,K}).\]
It is a perfect complex on $\LS_{G,K}$ of $\Tor$-amplitude $[-1,1]$ and Euler characteristic $\dim(G)\cdot\deg(K)$.
Thus, $\LS_{G,K}$ is a quasi-smooth dg-stack of expected dimension 
\[\edim\LS_{G,K}=\dim(G)\cdot\deg(K).\]

Let 
\[\pi:\LS_{G,K}\to\Bun_G:(\cF,\nabla)\mapsto\cF\] be the natural morphism. Put 
\[\LS_{G,K}^d:=\pi^{-1}(\Bun_G^d)=\{(\cF,\nabla)\in\LS_{G,K}:\deg(\cF)=d\}\qquad (d\in\pi_1(G)).\]
It is easy to see that $\pi$ is schematic (and, in an appropriate sense, $\pi$ is a `dg-affine bundle'). 

\subsection{Moduli of opers}
Denote the stack of degenerate $K$-opers $(\cF,\nabla,\cF_B)$ by $\Op_{G,K}$. Again, $\Op_{G,K}$ is a dg-stack;
in the course of the proof of Theorem~\ref{th:mainbundle}, 
we construct points $(\cF,\nabla,\cF_B)\in\Op_{G,K}$ where the stack is classical.

Let $(\cF,\nabla,\cF_B)$ be a degenerate $K$-oper. The degree $\deg(\cF_B)$ is an element of 
the cocharacter lattice $\Hom(\gm,T)=\pi_1(B)=\pi_1(T)$. Put
\[\Op_{G,K}^\delta:=\{(\cF,\nabla,\cF_B):\deg(\cF_B)=\delta\}\qquad(\delta\in\pi_1(T)).\] 

The tangent complex to $\Op_{G,K}$ is given by 
\[T_{(\cF,\nabla,\cF_B)}\Op_{G,K}=R\Gamma(X,\frb_{\cF_B}\to\ffg{-1}_{\cF_B}\otimes K)[1]\qquad((\cF,\nabla,\cF_B)\in\Op_{G,K}).\]
It is a perfect complex on $\Op_{G,K}$ of $\Tor$-amplitude $[-1,1]$;
its restriction to $\Op_{G,K}^\delta$ has Euler characteristic
\[\dim(\frb)\deg(K)+\sum_{\alpha\in S}(1-g-\langle \alpha,\delta\rangle+\deg(K)).\]
(Recall that $S\subset\Hom(T,\gm)$ is the set of the simple roots of $G$.)
Thus, $\Op_{G,K}^\delta$ is a quasi-smooth dg-stack of expected dimension 
\[\edim\Op_{G,K}^\delta=\dim(\frb)\deg(K)+\sum_{\alpha\in S}(1-g-\langle \alpha,\gamma\rangle+\deg(K))\qquad(\delta\in\pi_1(T)).\]

Let
\[\sfv:\Op_{G,K}\to\LS_{G,K}:(\cF,\nabla,\cF_B)\mapsto(\cF,\nabla)\]
be the natural projection. 
The map $\sfv$ is schematic. Let $T(\Op_{G,K}/\LS_{G,K})$ be the relative tangent complex of $\sfv$;
its fiber  
at $(\cF,\nabla,\cF_B)\in\Op_{G,K}$ is given by
\[T_{(\cF,\nabla,\cF_B)}(\Op_{G,K}/\LS_{G,K})=R\Gamma((\frg/\frb)_{\cF_B}\to(\frg/\ffg{-1})_{\cF_B}\otimes K).\]
The complex $T(\Op_{G,K}/\LS_{G,K})$ is a perfect of $\Tor$-amplitude $[0,2]$. 

\subsection{Plan of proof}
We can now reformulate Theorem~\ref{th:mainbundle} as follows:

\begin{theorem*}[Reformulation of Theorem~\ref{th:mainbundle}] 
The morphism $\sfv:\Op_{G,K}\to\LS_{G,K}$ is surjective.
\end{theorem*}
Its proof relies on the following propositions.

\begin{proposition} \label{prop:closedness}
The image of the map $\sfv:\Op_{G,K}\to\LS_{G,K}$ is specialization-closed.
\end{proposition}

\begin{proposition} \label{prop:good}
Suppose $\deg(K)>\max(2g-2,0)$, and $d\in\pi_1(G)$. Then 
\begin{enumerate}
\item $\LS_{G,K}^d$ is a classical local complete intersection stack;
\item $\LS_{G,K}^d$ is irreducible of dimension $\dim(G)\cdot\deg(K)$; 
\item The morphism $\pi:\LS_{G,K}^d\to\Bun^d_G$ is dominant.
\end{enumerate}
\end{proposition}

\begin{proposition} \label{prop:witness}
Suppose $\lambda_K:\Omega_X\to K$ is equal to zero, $d\in\pi_1(G)$, and $H^0(X,K)\ne 0$. 
Then there exists a degenerate oper $(\cF,A,\cF_B)\in\Op_{G,K}$ such that
\begin{enumerate}
\item $\deg\cF=d$;
\item The complex
\[T_{(\cF,A,\cF_B)}(\Op_{G,K}/\LS_{G,K})=R\Gamma((\frg/\frb)_{\cF_B}\overset{\ad(A)}\to
(\frg/\ffg{-1})_{\cF_B}\otimes K)\]
has cohomology in degree $0$ only.
\end{enumerate}
\end{proposition}

\begin{remarks*} 
The second condition of Proposition~\ref{prop:witness} is equivalent to smoothness of the morphism $\sfv$ at the
point $(\cF,A,\cF_B)\in\Op_{G,K}$. (It is important here that we work with dg stacks.)

Also, in Proposition~\ref{prop:witness}, $\lambda_K=0$, so that $K$-connections
are actually ($K$-valued) Higgs bundles. We use the letter $A$ instead of $\nabla$ to distinguish Higgs fields from more
general $K$-connections. 
\end{remarks*} 

The proofs of Proposition~\ref{prop:closedness}, \ref{prop:good}, and \ref{prop:witness} occupy Sections~\ref{sc:closedness}, \ref{sc:good} and \ref{sc:witness}, respectively. It is easy to see that the propositions imply Theorem~\ref{th:mainbundle}.

\begin{proof}[Proof of Theorem~\ref{th:mainbundle}] Let us show that the morphism $\sfv:\Op_{G,K}\to\LS_{G,K}$ is surjective. It suffices to prove the claim for sufficiently positive $K$, so we may assume that
$\deg(K)>\max(0,2g-2)$. This implies $H^0(X,K)\ne 0$, so $K$ satisfies the hypotheses of Propositions~\ref{prop:good} and \ref{prop:witness}.

By Proposition~\ref{prop:closedness}, it suffices to show $\sfv(\Op_{G,K})$ 
contains a dense open set (that is, that its complement is nowhere dense). By Proposition~\ref{prop:good}, it suffices to check that for every $d\in\pi_1(G)$, the intersection
\[\sfv(\Op_{G,K})\cap\LS_{G,K}^d\]
has non-empty interior. Since smooth morphisms are open, it is enough to check the following condition:
\begin{itemize}
\item[($\dagger$)] For every $d\in\pi_1(G)$, there exists 
a degenerate oper $(\cF,\nabla,\cF_B)\in\Op_{G,K}$ such that $\sfv$ is smooth at $(\cF,\nabla,\cF_B)$ and
$\deg(\cF)=d$. 
\end{itemize}

Let us now allow the map $\lambda_K:\Omega_X\to K$ to vary in the vector space $H^0(X,K\otimes\Omega_X^{-1})$.
Put \[\Lambda_\dagger:=\{\lambda_K\in H^0(X,K\otimes\Omega_X^{-1}):(\dagger)\text{ holds}\}.\]
It is easy to see that $\Lambda_\dagger$ is an open conical subset. By Proposition~\ref{prop:witness}, $0\in\Lambda_\dagger$, and therefore $\Lambda_\dagger=H^0(X,K\otimes\Omega_X^{-1})$, as required.
\end{proof}

\section{Drinfeld's structures and opers}\label{sc:closedness}

Propositions~\ref{prop:closedness} easily follows from looking at Drinfeld's compactification of 
the morphism $\Bun_B\to\Bun_G$. Let us sketch the argument.

\subsection{Drinfeld's structures on $G$-bundles}
Recall the definition of Drinfeld's compactification of the morphism $\Bun_B\to\Bun_G$; we follow Section~1 of 
\cite{BG}. As before, $N\subset B$ is the unipotent radical, so that $T=B/N$ is the Cartan group.

\begin{definition} Let $S$ be a scheme. By definition, an \emph{$S$-family of Drinfeld's structures} is a triple
$(\cF,\cF_T,\kappa)$, where 
\begin{itemize}
\item $\cF$ is a $G$-bundle on $S\times X$;

\item $\cF_T$ is a $T$-bundle on $S\times X$;

\item $\kappa$ is a collection of morphisms 
\[\kappa_V:(V^N)_{\cF_T}\to V_\cF\]
for all finite-dimensional representations $V$ of $G$.
\end{itemize}
satisfying the following two conditions:
\begin{enumerate}
\item For any $V$, the zero locus of $\kappa_V$ is finite over $S$;

\item The maps $\kappa_V$ satisfy the Pl\"ucker relations: for the trivial representation $V$, $\kappa_V$ is the identity, and $\kappa_V$ is compatible with morphisms of representations $V_1\otimes V_2\to V$.
\end{enumerate}
\end{definition}

Let $\overline{\Bun}_B$ be the stack whose category of $S$-points is the category of $S$-families of Drinfeld's structures. The stack $\overline{\Bun}_B$ is equipped with the natural map
\[\overline{\mathsf{p}}:\overline{\Bun}_B\to\Bun_G:(\cF,\cF_T,\kappa)\mapsto\cF.\]

\begin{proposition}[{\cite[Proposition~1.2.2]{BG}}] \label{prop:compactification}
$\overline{\Bun}_B$ is an algebraic stack. The map $\overline{\mathsf{p}}$ is representable,
and $\overline{\Bun}_B$ is a countable disjoint union of stacks that are proper over $\Bun_G$. \qed
\end{proposition}

\begin{proposition} \label{prop:nonsingBstruct}
Let $(\cF,\cF_T,\kappa)$ be an $S$-family of Drinfeld's structures.
\begin{enumerate}
\item\label{eq:nonsing} 
There is an open subset $U\subset S\times X$ such that the intersection $U\cap \{s\}\times X$ is non-empty
for any point $s\in S$ and that $\kappa^V$ has no zeroes on $U$ for any $V$;
\item Over such an open set $U$, the triple $(\cF,\cF_T,\kappa)$ is induced by a $B$-bundle $\cF'_B$.
This means that $\cF|_U=G_{\cF'_B}$ and $(\cF_T)|_U=T_{\cF'_B}$ are the induced bundles, and that for any 
$V$, $(\kappa^V)|_U$ is the natural map.
\item\label{eq:nonsingfield} Suppose $S$ is the spectrum of a field. Then the $B$-structure $\cF'_B$ on $\cF|_U$
can be extended to a $B$-structure $\cF_B$ on the $\cF$ over the entire $S\times X$. 
\end{enumerate}
\end{proposition}
\begin{proof} This is a version of \cite[Proposition~1.25]{BG}, which can be proved in the same way.
\end{proof}

\subsection{Drinfeld's oper structures}
We can now modify the notion of degenerate oper by replacing $B$-structures with Drinfeld's structures. 
Consider the stack
\[\LS_{G,K}\times_{\Bun_G}\overline{\Bun}_B.\] 
Given a scheme $S$, the category $\left(\LS_{G,K}\times_{\Bun_G}\overline{\Bun}_B\right)(S)$ is the 
category of collections $(\cF,\nabla,\cF_T,\kappa)$, where $(\cF,\nabla)\in\LS_{G,K}(S)$ is an $S$-family
of $K$-connections, and $(\cF,\cF_T,\kappa)$ is an $S$-family of Drinfeld's structures.

Given $(\cF,\nabla,\cF_T,\kappa)\in\left(\LS_{G,K}\times_{\Bun_G}\overline{\Bun}_B\right)(S)$, we apply
Proposition~\ref{prop:nonsingBstruct} to $(\cF,\cF_T,\kappa)$. We thus obtain an open subset $U\subset S\times X$
and a $B$-structure $\cF'_B$ on $\cF|_U$. We say that $(\cF,\nabla,\cF_T,\kappa)$ is an \emph{$S$-family of Drinfeld's oper
structures} (or, more precisely, degenerate oper structures) if the $B$-structure $\cF'_B$ satisfies 
Definition~\ref{def:oper} on $U$. It is easy to see that the condition does not depend on the choice of an open
set $U$ satisfying Proposition~\ref{prop:nonsingBstruct}\eqref{eq:nonsing}.

Denote by 
\[\overline{\Op}_{G,K}\subset \LS_{G,K}\times_{\Bun_G}\overline{\Bun}_B\] 
the stack parametrizing Drinfeld's oper structures. Its properties are summarized in the following easy proposition:

\begin{proposition}\label{prop:Dopers}
\begin{enumerate}
\item\label{eq:opers closed} $\overline{\Op}_{G,K}\subset \LS_{G,K}\times_{\Bun_G}\overline{\Bun}_B$ 
is a closed substack.

\item\label{eq:equal images}
 Consider the morphism 
\[\overline\sfv:\overline{\Op}_{G,K}\to\LS_{G,K}:(\cF,\nabla,\cF_T,\kappa)\mapsto(\cF,\nabla).\]
Then $\overline\sfv(\overline{\Op}_{G,K})=\sfv(\Op_{G,K})$. (Note that this is an equality between sets of points;
we do not consider any kind of algebraic structure on the images.)  
\end{enumerate}
\end{proposition}
\begin{proof}
\eqref{eq:opers closed} Without loss of generality, we may assume that the set $U\subset S\times X$ is the complement
of a divisor (see \cite[Lemma~3.2.7]{Bar}); in this case, the claim is easy.

\eqref{eq:equal images} Since $\sfv$ decomposes as
\[\Op_{G,K}\hookrightarrow\overline\Op_{G,K}\overset{\overline\sfv}\to\LS_{G,K},\]
the inclusion $\sfv(\Op_{G,K})\subset\overline\sfv(\overline\Op_{G,K})$ is clear. In the other direction,
let $S$ be the spectrum of a field, and suppose 
\[(\cF,\nabla,\cF_T,\kappa)\in\overline\Op_{G,K}(S),\]
so that $(\cF,\nabla)\in\overline\sfv(\overline\Op_{G,K})$. Then Proposition~\ref{prop:nonsingBstruct}\eqref{eq:nonsingfield} provides a $B$-structure $\cF_B$ on $\cF$, and it is
clear that $(\cF,\nabla,\cF_B)\in\Op_{G,K}(S)$. Therefore,
$(\cF,\nabla)\in\sfv(\overline\Op_{G,K})$, as required.
\end{proof}

\subsection{Proof of Proposition~\ref{prop:closedness}} We are now ready to prove the proposition.

\begin{proof}
By Proposition~\ref{prop:Dopers}\eqref{eq:equal images}, we need to show that the image $\overline\sfv(\overline\Op_{G,K})\subset\LS_{G,K}$ is specialization-closed. It follows from Proposition~\ref{prop:compactification} and Proposition~\ref{prop:Dopers}\eqref{eq:opers closed} that $\overline\Op_{G,K}$ is a countable disjoint union of stacks that are proper over $\LS_{G,K}$, and therefore
$\overline\sfv(\overline\Op_{G,K})$ is a countable union of closed sets.
\end{proof}

\section{Components of the stack of $K$-connections}\label{sc:good}

In this section, we prove Proposition~\ref{prop:good}. 

\subsection{Summary of results}
The line bundle $K$ defines the following locally closed subsets of $\Bun_G$:
\[
\cH_k:=\{\cF\in\Bun_G:\dim H^1(X,\frg_\cF\otimes K)=k\}\subset\Bun_G\qquad (k\ge 0).
\]

Consider the following two conditions on the triple $(X,G,K)$:

\begin{align}
\codim\cH_k\ge k&\qquad\text{for any }k>0\label{cn:goodline}\\
\codim\cH_k>k&\qquad\text{for any }k>0\label{cn:vgoodline}.
\end{align}

\begin{remark} The conditions \eqref{cn:vgoodline} and \eqref{cn:goodline}
are monotone: if the triple $(X,G,K)$  satisfies \eqref{cn:vgoodline} (resp. \eqref{cn:goodline}),
then so does $(X,G,K(D))$ for any divisor $D\ge0$.
\end{remark}

\begin{proposition}\label{prop:whatisgood}
\begin{enumerate}
\item If the condition \eqref{cn:goodline} is satisfied, $\LS_{G,K}$ is a classical stack; it is a 
local complete intersection of pure dimension \[\dim\LS_{G,K}=\deg(K)\cdot\dim(G).\]

\item If the stronger condition \eqref{cn:vgoodline} is satisfied, $\pi$ is dominant: for every $d\in\pi_1(G)$,
$\LS_{G,K}$ is irreducible and the map $\pi:\LS^d_{G,K}\to\Bun_G^d$ is dominant.
\end{enumerate}
\end{proposition} 
\begin{proof} Note that for any $\cF\in\cH^k$ either $\pi^{-1}(\cF)=\emptyset$, or $\dim\pi^{-1}(\cF)=(\deg(K)+1-g)\cdot\dim(G)+k$; thus,
\[\dim(\pi^{-1}(\cH_k))\le\deg(K)\cdot\dim(G)+k-\codim\cH_k.\] Moreover, the restriction 
\[\pi:\pi^{-1}(\cH_0)\to\cH_0\]
is a (classical) affine bundle of relative dimension $\deg(K)\cdot\dim(G)+k-\codim\cH_k$. Therefore,
under the condition \eqref{cn:goodline}, 
\[\dim\LS_{G,K}\le\deg(K)\cdot\dim G=\edim\LS_{G,K}.\]
If moreover \eqref{cn:vgoodline} holds, then $\pi^{-1}(\cH_0)\subset\LS_{G,K}$ is dense. 
Proposition~\ref{prop:whatisgood} follows.
\end{proof}

\begin{proposition} \label{prop:positiveisgood}
Suppose that 
\[\deg(K)>\max(0,2g-2).\]
Then \eqref{cn:vgoodline}
holds.
\end{proposition}

Clearly, Propositions~\ref{prop:whatisgood} and \ref{prop:positiveisgood} imply Proposition~\ref{prop:good}. We now proceed to prove Proposition~\ref{prop:positiveisgood}.

\begin{remark} In \cite[Section~1.1]{QHitch}, A.~Beilinson and V.~Drinfeld introduce `good' and `very good' properties of a smooth algebraic stack. If $K=\Omega_X$ is the canonical line bundle, 
the conditions \eqref{cn:goodline} and \eqref{cn:vgoodline} are equivalent to the good property and the 
very good property of $\Bun_G$, respectively.

By \cite[Proposition~2.1.2]{QHitch}, $\Bun_G$ is good if $G$ is semisimple and $g(X)>1$; that is,
the triple $(X,G,\Omega_X)$ satisfies \eqref{cn:vgoodline} under these assumptions. 
From this, it is not hard to see that the triple $(X,G,\Omega_X(D))$ 
satisfies \eqref{cn:vgoodline} if $G$ is reductive, $D>0$, and $g(X)>1$. 

We do not use \cite[Proposition~2.1.2]{QHitch} in the proof of Proposition~\ref{prop:positiveisgood};
our argument
is actually closer to the proof of the $g=1$ case of \cite[Theorem~2.10.4]{QHitch}. 
\end{remark}

\subsection{The Shatz stratification of $\Bun_G$}
We proceed by using the Shatz stratification of $\Bun_G$ (also known as the Harder-Narasimhan stratification). Our 
conventions mostly follow the proof of the $g=1$ case of Theorem~2.10.4 of \cite{QHitch}.

Let $P\subset G$ be a parabolic subgroup; denote by $U\subset P$ its unipotent radical and by
$L=P/U$ its maximal reductive quotient. We have natural maps $p:P\hookrightarrow G$ and $q:P\to L$.
Let $T\subset P$ be a maximal torus.
Consider the character lattices
\[\Hom(L,\gm)=\Hom(P,\gm)\subset\Hom(T,\gm).\]

Denote by $\Lambda_U\subset\Hom(T,\gm)$ the set of weights of the adjoint action of $T$ on the 
Lie algebra $\fru$ of $U$. Define the projection
\[\sigma:\Hom(T,\gm)\to\Hom(L,\gm)\otimes\Q\]
to be the composition
\[\sigma:\Hom(T,\gm)\to\Hom(Z(L),\gm)\hookrightarrow\Hom(Z(L),\gm)\otimes\Q\simeq\Hom(L,\gm)\otimes\Q;\]
here $Z(L)$ is the center of $L$.

Put 
\begin{align*}
X_P&:=\Hom(P,\gm)^\vee=\Hom(\Hom(P,\gm),\Z),\\
X_P^+&:=\{d\in\Hom(P,\gm)^\vee:\langle d,\sigma(\phi)\rangle>0\text{ for any $\phi\in\Lambda_U$}\}\subset X_P.
\end{align*}
(Note that $X_P$ is the quotient of $\pi_1(P)$ by its torsion.) Recall that the degree of a $P$-bundle $\cF$ is the element
$\deg(\cF)\in X_P$ such that
\[\langle \deg(\cF),\phi\rangle=\deg(\phi_*(\cF))\qquad\text{for any $\phi\in\Hom(P,\gm)$}.\]
Fix $d\in X_P^+$ and put
\[\Shatz_P^d:=\{\cF\in\Bun_P:\deg(\cF)=d\text{ and }q_*(\cF)\in\Bun_L\text{ is semistable}\}\subset\Bun_P.\]
It is known that the map $p_*:\Bun_P\to\Bun_G$ restricts to a locally closed embedding 
\[\Shatz_P^d\hookrightarrow\Bun_G,\]
and that the images of these maps over all conjugacy classes of pairs $(P,d\in X_P^+)$ provide a
stratification of $\Bun_G$. We identify $\Shatz_P^d$ with its image in $\Bun_G$.

Direct calculation gives
\[\dim\Shatz_P^d=\dim(P)(g-1)-\sum_{\phi\in\Lambda_P}\langle d,\sigma(\phi)\rangle.\] 
Hence, the codimension of $\Shatz_P^d\subset\Bun_G$ is
\[\codim\Shatz_P^d=\sum_{\phi\in\Lambda_P} \left(\langle d,\sigma(\phi)\rangle+(g-1)\right).\]

\subsection{Proof of Proposition~\ref{prop:positiveisgood}}
Let $K':=K^\vee\otimes\Omega_X$ be the Serre dual of $K$. Consider the stack of $K'$-valued Higgs bundles 
(equivalently, bundles with $K'$-connections for the zero map $\lambda_{K'}:\Omega_X\to K'$):
\[\LS_{G,K'}=\{(\cF,A):\cF\in\Bun_G,A\in H^0(X,\frg_\cF\otimes K')\}.\]
By Serre's duality, the fiber of the projection $\pi':\LS_{G,K'}\to\Bun_G$ over $\cF\in\Bun_G$ is isomorphic to $H^1(X,
\frg_\cF\otimes K)^\vee$. Hence, we need to show that $\dim(\LS_{G,K'})=\dim(\Bun_G)$ and that each component of $\LS_{G,K'}$ of 
maximal dimension is contained in the locus $A=0$.

Fix a parabolic subgroup $P\subset G$ and $d\in X_P^+$, and put
\[\Shatz'_{P,d}:=\Shatz^P_d\times_{\Bun_G}\LS_{G,K'}.\]
It suffices to check that the following two statements hold for any $P$ and $d$:

\begin{enumerate}
\item If $\codim(\Shatz^P_d)=0$ (so that $\Shatz^P_d\subset\Bun_G$ is an open stratum), then the projection
$\Shatz'_{P,d}\to\Shatz^P_d$ is an isomorphism. In other words, if $(\cF,A)\in\LS_{G,K'}$ and $\cF\in\Shatz^P_d$ for
$\codim(\Shatz^P_d)=0$, then $A=0$. 
\item If $\codim(\Shatz^P_d)>0$, then $\dim(\Shatz'_{P,d})<\dim(\Bun_G)$.
\end{enumerate} 

Fix $\cF\in\Shatz_{P,d}$. Note that the subbundle $\fru_\cF\subset\frg_\cF$ is one of the terms of 
the Harder-Narasimhan filtration on $\frg_\cF$; it is uniquely determined by the property that $\fru_\cF$ is an iterated extension
of semistable bundles of positive degree, while $\frg_\cF/\fru_\cF$ is an iterated extension of semistable bundles of non-positive degree. (As above, $U\subset P$ is the unipotent radical and $\fru$ is its Lie algebra.)
Since $\deg(K')<0$, we see that $H^0(X,\frg_\cF\otimes K')=H^0(X,\fru_\cF\otimes K')$.

{\it Proof of (1).} Suppose $\codim(\Shatz^P_d)=0$. This occurs if and only if either $P=G$ or $g=0$ and
$\fru_\cF\simeq\cO_X(1)^{\dim\fru}$. In either case, we see that  
\[H^0(X,\frg_\cF\otimes K')=H^0(X,\fru_\cF\otimes K')=0\]
for any $\cF\in\Shatz^P_d$, as required. (Recall that if $g=0$, we have $\deg(K)>0$, and therefore $\deg(K')<-2$.) 

{\it Proof of (2).} Suppose $\codim(\Shatz^P_d)>0$. Note the following estimate.

\begin{lemma} Let $E$ be a semistable vector bundle on $X$ of slope at least $-1$. Then 
\[\dim H^0(X,E)\le\deg(E)+\rk(E).\]
\label{lm:ssH0}
\end{lemma}
\begin{proof} Take an effective divisor $D$ of degree $\deg(D)=\left\lfloor\dfrac{\deg(E)}{\rk(E)}\right\rfloor+1$
and use injectivity of the map
\[H^0(X,E)\to H^0(X,E/E(-D)).\]
\end{proof}

{\it Case I:} $g>1$. It suffices to prove that 
\[\dim(H^0(X,\fru_\cF\otimes K'))<\codim\Shatz_P^d\]
for any $\cF\in\Shatz_P^d$.
Since $\deg(K')<0$, and the right-hand side is independent of $K'$, we may assume that $\deg(K')=-1$. Applying 
Lemma~\ref{lm:ssH0} to the semistable sub-quotients of $\fru_\cF\otimes K'$, we see that
\begin{multline*}
\dim H^0(X,\fru_\cF\otimes K')\le \deg(\fru_\cF\otimes K')+\rk(\fru_\cF\otimes K')=\deg(\fru_\cF)\\
=\sum_{\phi\in\Lambda_U}\langle d,\sigma(\phi)\rangle<
\sum_{\phi\in\Lambda_U} \left(\langle d,\sigma(\phi)\rangle+(g-1)\right)=\codim\Shatz_P^d,
\end{multline*}
as claimed.

{\it Case II:} $g=1$. The argument used in Case I now gives only the non-strict inequality
\[\dim\Shatz'_{P,d}\le\dim(\Bun_G).\]
However, it is easy to see that the inequality is strict for generic $K'$ of given degree $k<0$. 
Since the automorphisms of the elliptic curve $X$ act transitively on $\Pic^k(X)$ for $k\ne 0$, we see that the
inequality is in fact strict for all $K'$.

{\it Case III:} $g=0$. This case easily follows from the explicit formula  
\[\fru_\cF\simeq\bigoplus_{\phi\in\Lambda_U}\cO_X(\langle d,\sigma(\phi)\rangle).\]

\section{Oper Higgs bundles}\label{sc:witness}
In this section, we prove Proposition~\ref{prop:witness}. Recall that we are given $d\in\pi_1(G)$
and a line bundle $K$ on $X$ satisfying $H^0(X,K)\ne 0$. We need to construct a triple
$(\cF,A,\cF_B)$, where $\cF$ is a $G$-bundle, $A\in H^0(X,\frg_\cF\otimes K)$ is a $K$-valued Higgs field on $\cF$,
and $\cF_B$ is a $B$-structure on $\cF$. The triple $(\cF,A,\cF_B)$ must satisfy the following conditions:
\begin{itemize}
\item $\deg(\cF)=d$;
\item $A$ is a degenerate oper with respect to $\cF_B$, that is, $A\in H^0(X,\ffg{-1}_{\cF_B}\otimes K)$;
\item The complex \[R\Gamma((\frg/\frb)_{\cF_B}\overset{\ad(A)}\to(\frg/\frg^{-1})_{\cF_B}\otimes K)\]
has cohomology in degree zero only.
\end{itemize}

The construction of the triple $(\cF,\nabla,\cF_B)$ proceeds in
two steps. 

\subsection{Proof of Proposition~\ref{prop:witness}: first step} 
Fix a maximal torus $T\subset B$.

Choose a section $A_0\in H^0(X,\frt\otimes_{\kk} K)$ that is generically regular as
a section of $\frg\otimes_{\kk}K$. Suppose $A_0$ is regular at all points of 
$X-\Sigma$, where $\Sigma\subset X$ is a finite nonempty subset.

Now choose a $T$-bundle $\cF'_T$ that is sufficiently negative in the following way: for each simple root 
$\alpha\in S$, there exists a section $A_{-\alpha}\in H^0(X,(\frg_{-\alpha})_{\cF'_T}\otimes K)$ that does not vanish on $\Sigma$.
Let us fix such sections.

Consider now the induced $G$-bundle $\cF':=G_{\cF'_T}$ and equip it with the Higgs field 
\[A=A_0+\sum_{\alpha\in S}A_{-\alpha}\in\frg_{\cF'_T}\otimes K.\]
This Higgs field is a (generically non-degenerate) oper with respect to the natural $B$-structure $\cF'_B:=B_{\cF'_T}$ on $\cF'$.
We thus have the complex
\[R\Gamma((\frg/\frb)_{\cF'_B}\overset{\ad(A)}\to(\frg/\frg^{(-1)})_{\cF'_B}\otimes K).\]

\begin{lemma} \label{lm:adsur}
The morphism of vector bundles
\[\ad(A):(\frg/\frb)_{\cF'_B}\to(\frg/\frg^{(-1)})_{\cF'_B}\otimes K\]
is surjective.
\end{lemma}
\begin{proof} At points of $\Sigma$, the map is surjective because it is surjective on the associated graded
of the filtrations
\[\xymatrix@C=2em{
0\ar@{}[rr]|*+{\subset}&&
(\ffg{-1}/\frb)_{\cF'_B}\ar[d]\ar@{}[rr]|*+{\subset\dots\subset}&&
(\ffg{-k}/\frb)_{\cF'_B}\subset\dots\ar[d]\\
0\ar@{}[rr]|*+{\subset}&&
(\ffg{-2}/\ffg{-1})_{\cF'_B}\otimes K\ar@{}[rr]|*+{\subset\dots\subset}&&
(\ffg{-k-1}/\ffg{-1})_{\cF'_B}\otimes K\subset\dots
}\]

At points of $X-\Sigma$, the map is surjective because it is surjective on the associated graded
of the filtrations
\[\xymatrix@C=2em{
{}\ar@{}[rr]|*+{\dots\subset}&&
(\frg^{[k]})_{\cF'_T}\ar[d]\ar@{}[rr]|*+{\subset\dots\subset}&&
(\frg^{[2]})_{\cF'_T}\ar[d]\ar@{}[r]|*+{\subset}&
(\frg^{[1]})_{\cF'_T}=(\frg/\frb)_{\cF'_B}\\
{}\ar@{}[rr]|*+{\dots\subset}&&
(\frg^{[k]})_{\cF'_T}\otimes K\ar@{}[rr]|*+{\subset\dots\subset}&&
(\frg^{[2]})_{\cF'_T}\otimes K\ar@{}[r]|*+{=}&
(\frg/\ffg{-1})_{\cF'_B}\otimes K.
}\]
Here $\frg^{[k]}$ is the filtration of $\frg$ with respect to the opposite Borel sublagebra $\frb^-$,
which is defined recursively by 
\begin{align*}
\frg^{[0]}&=\frb^-\\ 
\frg^{[1]}&=[\frb^-,\frb^-]\\
\frg^{[k+1]}&=[\frg^{[k]},\frg^{[1]}]&(k>0)
\end{align*}
\end{proof}
Put
\[E':=\ker\left(\ad(A):(\frg/\frb)_{\cF'_B}\to(\frg/\ffg{-1})_{\cF'_B}\otimes K\right).\]

\subsection{Proof of Proposition~\ref{prop:witness}: second step}
Pick a point $x\in X-\Sigma$. There exists $N\ge 0$ such that $H^1(X,E'(Nx))=0$. Now pick a cocharacter
$\gamma:\gm\to T$ with the following properties:

\begin{itemize}
\item For any simple root $\alpha\in S$, we have $\langle\gamma,\alpha\rangle\le -N$.
\item The image of the sum $\gamma+\deg(\cF'_T)\in\Hom(\gm,T)$ in $\pi_1(G)$ is the prescribed degree $d$.
(Recall that $\pi_1(G)$ is isomorphic to the quotient of the cocharacter lattice $\Hom(\gm,T)=\pi_1(T)$ by the 
coroot lattice.)
\end{itemize}

Now let $\cF_T$ be the modification of $\cF'_T$ at $x$ corresponding to the cocharacter $\gamma$. That is, 
\[\cF_T=(\cF'_T\times_X\cO(x)^\times)/\gm,\]
where $\cO(x)^\times$ is the complement of the zero section in the total space of the line bundle $\cO_X(x)$, and
the action of $\gm$ on $\cF'_T\times_X\cO(x)^\times$ is given by
\[a\cdot(s,f)=(\gamma(a)s,a^{-1}f)\qquad a\in\gm,s\in\cF'_T,f\in\cO(x)^\times.\]

Let $\cF:=G_{\cF_T}$ be the induced $G$-bundle. $\cF$ is equipped with the
natural $B$-structure $\cF_B:=B_{\cF_T}$.

By construction, we have an identification $\cF'|_{X-x}=\cF_{X-x}$, which induces a rational map of vector bundles
\[
\iota:\frg_{\cF'}\dashrightarrow\frg_\cF.
\]
The choice of $\gamma$ implies the following property of $\iota$:
\begin{lemma} $\iota$ induces a regular map 
\[(\frg^{[k]})_{\cF'_T}\otimes\cO((kN)x)\to(\frg^{[k]})_{\cF_T}.\]
\qed
\end{lemma}

In particular, $\iota$ gives a regular map $(\frb^-)_{\cF'_T}\to(\frb^-)_{\cF'}$. This allows us to view 
\[A\in H^0(X,(\frb^-)_{\cF'_T}\otimes K)\subset H^0(X,\frg_{\cF'}\otimes K)\]
as a Higgs field on $\cF$:
\[A\in H^0(X,(\frb^-)_{\cF_T}\otimes K)\subset H^0(X,\frg_\cF\otimes K).\]
We claim that the triple $(\cF,A,\cF_B)$ satisfies the conditions of Proposition~\ref{prop:witness}. The first
condition is clear from the construction. It remains to show that the complex
\[R\Gamma((\frg/\frb)_{\cF_B}\overset{\ad(A)}\to(\frg/\frg^{(-1)})_{\cF_B}\otimes K)\]
has cohomology in degree zero only.

Lemma~\ref{lm:adsur} applies to the triple $(\cF,A,\cF_B)$, and therefore
\[R\Gamma((\frg/\frb)_{\cF_B}\overset{\ad(A)}\to(\frg/\frg^{(-1)})_{\cF_B}\otimes K)=R\Gamma(E),\]
where
\[E=\ker\left(\ad(A):(\frg/\frb)_{\cF_B}\to(\frg/\frg^{(-1)})_{\cF_B}\otimes K\right).\]
We need to verify that $H^1(X,E)=0$. However, using isomorphism $\frg/\frb=\frg^{[1]}$, we see that $\iota$
induces a regular injective map
\[(\frg/\frb)_{\cF'_B}\otimes\cO(Nx)\to(\frg/\frb)_{\cF_B}.\]
Since $E'\subset (\frg/\frb)_{\cF'_B}$ and $E\subset (\frg/\frb)_{\cF_B}$ are subbundles (that is, saturated subsheaves), and they agree on $X-\Sigma$, we see that $\iota$ also induces a regular injective map
\[E'(Nx)\to E.\]
Recall now that $H^1(X,E'(Nx))=0$; therefore, $H^1(X,E)=0$ as claimed.

This completes the proof of Proposition~\ref{prop:witness}.

\section{Opers on the formal disk}\label{sec:formaldisk}

This section contains the global-to-local argument that derives Theorem~\ref{th:FZ} from Theorem~\ref{th:main}.

\subsection{Connections on the formal disk}
Let us start with a few easy observations about connections on the punctured formal disk. The observations
follow from the reduction theory of formal connections due to D.~G.~Babbitt and V.~S.~Varadarajan \cite{BV}, but we were unable
to find a perfect reference; we therefore provide
proofs for the reader's convenience. 

As in Section~\ref{sec:intro},
suppose $\K=\kk((t))$, and put
\[\Conn_G(\K)=\{d+A(t):A(t)\in\frg\otimes_\kk\Omega_{\K/\kk}\}.\] 
For $\nabla_1,\nabla_2\in\Conn_G(\K)$, we write $\nabla_1\sim\nabla_2$ to indicate that the two connections are gauge-equivalent, that is, that \[\nabla_1=\Ad(g)\nabla_2\] for some $g\in G(\K)$. Put $\OO=\kk[[t]]$.

Recall that $\nabla\in\Conn_G(\K)$ is irreducible if whenever $\nabla\sim\nabla'$,
then 
\[\nabla'\not\in\Conn_P(\K)\]
for any proper parabolic subgroup $P\subset G$. It is easy to construct irreducible connections over $\K$: 

\begin{lemma} \label{lm:irred}
There exist irreducible $\nabla\in\Conn_G(\K)$.
\end{lemma}
\begin{proof} Here is one possible construction, which is a variation of the connection studied 
by E.~Frenkel and B.~Gross in \cite{FG}. Put
\[\nabla=d+A_{-2}t^{-2}dt+A_{-1}t^{-1}dt\]
for the following $A_{-2},A_{-1}\in\frg$:

\begin{itemize}
\item $A_{-2}\in\frb$ is a regular nilpotent
\item $A_{-1}\in\frg$ is such that its projection onto $\frg_{-\theta}$ is non-zero for any highest root $\theta$
(if $\frg$ is semisimple, the highest root is unique). Here $\frg_{-\theta}$ is the root space corresponding
to the root $-\theta$.
\end{itemize}

Let us show that $\nabla$ is irreducible. Indeed, suppose $\nabla'\sim\nabla$ and $\nabla\in\Conn_P(\K)$ 
for a proper parabolic subgroup $P\subset G$. Without loss of generality, we may assume that $P\supset B$. 
Choose $g\in G(\K)$ such that $\nabla'=\Ad(g)\nabla$. By the Iwasawa decomposition, we can write $g=b\cdot g'$
for $b\in B(\K)$ and $g'\in G(\OO)$. Let us replace $g$ with $g'$ and assume from the beginning that
$g\in G(\OO)$. 

Let us write \[g\equiv g_0\cdot (1+g_1t)\mod t^2\qquad(g_0\in G,g_1\in\frg).\] Then
\[\Ad(g)\nabla\equiv d+\Ad(g_0)A_{-2}t^{-2}dt+\Ad(g_0)(A_{-1}+[g_1,A_{-2}])t^{-1}dt\mod\Omega_{\OO/\kk}.\]
However, now the condition $\Ad(g)\nabla\in\Conn_P(\K)$ leads to a contradiction: it implies that
\[\begin{cases}\Ad(g_0)A_{-2}\in\frp\\ \Ad(g_0)(A_{-1}+[g_1,A_{-2}])\in\frp\end{cases}\]
the first condition yields $g_0\in P$, while in the second condition, $A_{-1}+[g_1,A_{-2}]\not\in\frp$. 
Here $\frp$ is the Lie algebra of $P$.
\end{proof}

\begin{lemma} \label{lm:gauge class is open}
The gauge-equivalence classes are open in $\Conn(\K)$: for any $\nabla\in\Conn(\K)$, there exists
$n$ such that whenever \[\nabla'-\nabla\in t^n\frg\otimes\Omega_{\OO/\kk},\] 
we have $\nabla'\sim\nabla$.
\end{lemma}
\begin{proof}
Fix $\nabla=d+A\in\Conn(\K)$, and let $d\ge 0$ be its order of pole,
so that \[A\in t^{-d}\frg\otimes_\kk\Omega_{\OO/\kk}.\] Now
consider the induced connection on the vector space $\frg\otimes_\kk\K$: 
\[\ad(\nabla):\frg\otimes_\kk\K\to\frg\otimes\Omega_{\K/\kk}.\]
There exists $N$ such that
\[\ad(\nabla)(t^k\frg[[t]])\supset t^{k+N}\frg\otimes\Omega_{\OO/\kk}\]
for any $k\ge 0$. (This is easy to see once $\ad(\nabla)$ is written in the Levelt-Turittin normal form.)

We can now choose $n=2N+d+2$. Indeed, if $a:=\nabla'-\nabla\in t^n\frg\otimes\Omega_{\OO/\kk}$, we can
choose $b\in t^{n-N}\frg\otimes\Omega_{\OO}$ such that \[\ad(\nabla)b=[A,b]+db=a.\] Pick $g\in G(\OO)$ such that
\[g\equiv 1-b\mod t^{2(n-N)}.\]
Then 
\begin{align*}
g^{-1}&\equiv 1+b\mod t^{2(n-N)}\\
dg&\equiv db\mod t^{2(n-N)-1}
\end{align*}
and therefore
\[\Ad(g)(\nabla)=d+\Ad(g)A-dg\cdot g^{-1}\equiv d+A-[b,A]+db\mod t^{\min(2(n-N)-d,2(n-N)-1)}.\]
Since $\min(2(n-N)-d,2(n-N)-1)\ge n+1$, we see that
\[\Ad(g)(\nabla)\equiv\nabla+a=\nabla'\mod t^{n+1}.\]
We can now iterate. 
\end{proof}

\subsection{From global to local opers}
Let $X$ be a smooth connected projective curve over $\kk$ and $\F=\kk(X)$. Fix $n$ distinct points $x_1,\dots,x_n\in X$, and
let $\K_i$ be the completion of $\F$ at $x_i$ ($i=1,\dots,n$). Lemma~\ref{lm:gauge class is open} 
implies the following statement:

\begin{corollary}\label{co:approx}
For any $n$ connections $\nabla_i\in\Conn_G(\K_i)$ ($i=1,\dots,n$),
there exists $\nabla\in\Conn_G(\F)$ such that $\nabla\sim\nabla_i$ in $\Conn_G(\K_i)$ for $i=1,\dots,n$.
\end{corollary}
\begin{proof} By Lemma~\ref{lm:gauge class is open}, it suffices to choose $\nabla$ that approximates $\nabla_i$
closely enough.
\end{proof}

Let us now show that Theorem~\ref{th:main} implies Theorem~\ref{th:FZ}.

\begin{proof}[Proof of Theorem~\ref{th:FZ}] Let $X$ be a smooth connected projective curve.
Fix $x_1,x_2\in X$, $x_1\ne x_2$. Let us show that any $\nabla_1\in\Conn_G(\K_1)$ admits an oper structure. Pick any irreducible $\nabla_2\in\Conn_G(K_2)$, which exists by Lemma~\ref{lm:irred}. By Corollary~\ref{co:approx}, there exists $\nabla\in\Conn_G(\F)$ such that $\nabla\sim\nabla_i$ in $\Conn_G(\K_i)$ ($i=1,2$).
Then $\nabla$ is irreducible because $\nabla_2$ is irreducible.
Now Corollary~\ref{co:operirreducible} implies that $\nabla\sim\nabla'$ for some $\nabla'\in\Oper_G(\F)$. Therefore,
$\nabla_1\sim\nabla'$, and $\nabla'\in\Oper_G(\K_1)$, as required.
\end{proof}

\bibliographystyle{abbrv}
\bibliography{opers}
\end{document}